\documentclass[english]{amsart}
\usepackage{amsthm}
\usepackage{amssymb}

\makeatletter
\numberwithin{equation}{section}
\numberwithin{figure}{section}
  \theoremstyle{definition}
  \newtheorem{defn}{\protect\definitionname}
\theoremstyle{plain}
\newtheorem{thm}{\protect\theoremname}
  \theoremstyle{definition}
\newtheorem{problem}{\protect\problemname}
  \theoremstyle{plain}
\newtheorem{lem}{\protect\lemmaname}
  \theoremstyle{remark}
\newtheorem{claim}{\protect\claimname}

\newtheorem*{corollary*}{Corollary}

\usepackage{xspace}
\newcommand{\dom}{\mathrm{dom }\xspace}
\newcommand{\im}{\mathrm{im }\xspace}

\usepackage{babel}
\providecommand{\claimname}{Claim}
\providecommand{\definitionname}{Definition}
\providecommand{\lemmaname}{Lemma}
\providecommand{\problemname}{Problem}
\providecommand{\theoremname}{Theorem}

\makeatother

\usepackage{babel}
  \providecommand{\claimname}{Claim}
  \providecommand{\definitionname}{Definition}
  \providecommand{\lemmaname}{Lemma}
  \providecommand{\problemname}{Problem}
\providecommand{\theoremname}{Theorem}

\begin{document}
%
%
\begin{abstract}
We construct new models of $ZF$ with an uncountable set of reals that
has a unique condensation point. This addresses a question by Sierpi\'{n}ski
from 1918. 
\end{abstract}

\author{Eilon Bilinsky\\
 Tel Aviv University}

\title{On uncountable strongly concentrated sets of reals}

\maketitle
\section{Introduction}

The real line is one of the most basic concepts in modern mathematics.
In particular, questions about the topological properties of $\mathbb{R}$
were studied extensively since the late $19^{\mathrm{th}}$ century
and the beginning of the $20^{\mathrm{th}}$ century. 
Some of those basic questions were settled easily using Choice principles, but remain
difficult in the absence of choice.
In the $20^{\mathrm{th}}$ century, in view of the independence phenomena in $\mathrm{ZF}$, people start asking about what knowledge the standard theories give us, in particular with respect to those questions about $\mathbb{R}$.

In this paper we focus on questions related to the existence of condensation
points of large subsets of the real line. Intuitively, since $\mathbb{R}$
is separable and complete, one might expect that any uncountable subset
of the real line will have more then one condensation points. Indeed, assuming
the Axiom of Choice, this is
provable.

In \cite{Sierpinski}, Sierpi\'{n}ski asked whether some from of Choice
is indeed required (see Problem \ref{problem} for exact formulation).
This question can be reformulated as a question about concentrated
sets. An uncountable set $X\subseteq\mathbb{R}$ has a single condensation
point if and only if it is concentrated on a singleton.

The goal of this paper is to give additional examples for models of $ZF$ which provide a positive answer for Sierpi\'{n}ski's question. 
Other models in which there are large bounded sets of reals which are concentrated on a single point can be also obtained using 
the Feferman-Levy method \cite{FefermanLevy1963}, 
or a variation of Cohen's method, \cite{Cohen1963}. 

In all these models there exists a bounded subset  $X\subseteq \mathbb{R}$ such that, in the model, $X$ has a unique condensation point. We will give two methods for obtaining a model in which such
a set exists. In the first method $X$ is well orderable (and therefore
by Lemma \ref{lem: singular aleph1}, $\aleph_{1}$ is singular)
and in the second method $\aleph_{1}$ is regular. Moreover, the models which are 
obtained in the second method are closer (in some sense) to the model of choice we start with. 

The structure of the paper is as follows. In Section \ref{section: preliminaries},
we will review some basic concepts and theorems which are relevant
for the question. In Section \ref{section: wo}, we will show that well orderable large 
strongly concentrated sets of reals exist in some of the Feferman-L\'{e}vy models. 
In Section \ref{section: regular aleph1} we will construct a model of $ZF$ in 
which there is a large strongly concentrated set of reals and $\aleph_1$ is regular. 

We work in $ZF$ and mention any use of the axiom of choice. Our notations are mostly standard. For basic facts about forcing and models with atoms (models of $ZFA$) we refer the reader to \cite{JechSetTheory}. 

\section{Acknowledgments}
I would like to thank
Lior Shalom, Michal Amir, Limor Friedman, Itamar Rosenfeld Rauch, Oren Yakir, Karina Samvelyan, Dor Elboim, Leonid Vishnevsky, Elad Zelingher, Peleg Michaeli, Ofir Gorodetsky, Asaf Cohen, Eyal Kaplan and Tom Benhamou for their help in the technical issues.
I would like to thank Heike Mildenberger for her remarks and encouragement.
I would like to thank William Chen, Assaf Rinot for reviewing a draft the paper. 
I would like to thank Asaf Kargila for pointing me to Sierpi\'{n}ski question. 
I would like to thank Yair Hayut for his help in the technical issues and for improving the style of the paper.
I would like to thank Moti Gitik for his guidance and  specific for his important suggestions.

Finally, I would like to thank the anonymous referee for their thorough reading of the paper and their critical suggestions that improved this paper significantly. 
In particular, their report broadened my historical and mathematical view and pointed me to some crucial issues regarding the topics of this paper, which I was not aware of.
\section{Preliminaries}
\label{section: preliminaries}
\begin{defn} 
Let us define a class function $\alpha\mapsto\aleph_\alpha$ by: 
For all ordinal $\alpha$, let us define $\aleph_\alpha$ to be the cardinal of the set of all ordinals
such that their cardinality is finite or equal to some $\aleph_\beta$ for some $\beta<\alpha$.	
\end{defn}
\begin{defn}
For all ordinal $\alpha$ let us define	
$\beth_\alpha=|V_{\omega+\alpha}|$.	
\end{defn}
\begin{defn}
For an ordinal $\alpha$ we define	
$\mathrm{cf}\aleph_\alpha$ to be the minimal $\aleph_\beta$ such that there exists a set $A$ of sets, such that the cardinality of each set in $A$ is less than $\aleph_\alpha$, $|A|=\aleph_\beta$ and $|\bigcup A|=\aleph_\alpha$. 
\end{defn}
For every $\aleph_\alpha$, $\mathrm{cf}(\aleph_\alpha)$ exists and $\mathrm{cf}(\aleph_\alpha)\le \aleph_\alpha$.

\begin{defn}
A singular cardinal is $\aleph_\alpha$ in which $\mathrm{cf}(\aleph_\alpha)< \aleph_\alpha$.
\end{defn}  
\begin{defn}
A regular cardinal is $\aleph_\alpha$ in which $\mathrm{cf}(\aleph_\alpha)= \aleph_\alpha$.
\end{defn}
The claim ``$\aleph_1$ is a regular cardinal'' is provable by the axiom of choice (\cite[Form 34]{TheDictionary}).  
\begin{defn}
For all set $X$ and an ordinal $\alpha$ let us define:
	
$P_{\aleph_\alpha}(X)=\{Y\subseteq X\mid |Y|<\aleph_\alpha\}$.
\end{defn}

\begin{defn}
D-infinite set is a set $A$ in which exist some $B\subset A$ ($B \ne A$)
such that $|B|=|A|$. 	
\end{defn}

\begin{defn}
D-finite set is a set $A$ such that $A$ is not D-infinite set
\end{defn}

A set $A$ is D-finite if and only if not exist an injection $f\colon\omega\to S$, namely $A$ has no infinite countable subset.

\begin{defn}
The axiom $CUT (\mathbb{R})$ is the axiom that for every set $A$
if $|A|=\aleph_0$ and every element in $A$ is a countable subset of $\mathbb{R}$ then $\bigcup A$ is countable.
\end{defn}

During this paper we will use the following convention:
\begin{defn}
A set $A$ is \emph{large} if and only if $A$ is not finite and not countable.
A set $A$ is \emph{uncountable} if and only if $A$ is large and D-infinite. 
\end{defn}

\begin{defn}
Let $A\subseteq\mathbb{R},r\in\mathbb{R}$. Then $r$ is a condensation
point of $A$ if and only if for every neighborhood $U$ of $r$ ,
$A\cap U$ is large. 
\end{defn}
The following classic definition is due to Besicovitch: 
\begin{defn}[Besicovitch, \cite{Besicovitch}]
A set $A\subseteq\mathbb{R}$ is \emph{concentrated} on a set $D\subseteq\mathbb{R}$
if and only if for every neighborhood $G$ of $D$, $\left|A\setminus G\right|\leq\aleph_{0}$. 
\end{defn}

The following definition will be central in this paper: 
\begin{defn}
A set $A\subseteq\mathbb{R}$ is \emph{strongly concentrated} if and
only if there is $c\in\mathbb{R}$ such that $A$ is concentrated
on the singleton $\{c\}$. 
\end{defn}

The following theorem is classical:
\begin{thm}[Existence of a condensation point]
\label{thm:Existence-of-a condensation point} For every bounded
$A\subseteq\mathbb{R}$, if $A$ is large then $A$ has at least one
condensation point. 
\end{thm}
Note that the proof does not use the Axiom of Choice. 
\begin{proof}
Since $A$ is bounded, there exist $a,b\in\mathbb{R}$ such that $A\subseteq\left[a,b\right]$..
Let us define the following two sequences $a_{n},b_{n}$: 
\begin{itemize}
\item $a_{0}=a$, $b_{0}=b$ 
\item $c_{n}=\frac{a_{n}+b_{n}}{2}$ 
\item If $A\cap\left[a_{n},c_{n}\right]$ is large, $a_{n+1}=a_{n}$, $b_{n+1}=c_{n}$.
Otherwise, $a_{n+1}=c_{n}$, $b_{n+1}=b_{n}$, 
\end{itemize}
Observe that for every $n\in\mathbb{N},a_{n}<b_{n}$, $A\cap\left[a_{n},b_{n}\right]$
is large. Also note that 
\[
b_{n+1}-a_{n+1}=\frac{1}{2}\left(b_{n}-a_{n}\right)
\]
and in particular, 
\[
\lim_{n\to\infty}\left(b_{n}-a_{n}\right)=0.
\]
Thus, from Cantor's lemma, there exists a unique point $c\in\bigcap_{n\in\mathbb{N}}\left[a_{n},b_{n}\right]$.

Let us claim that $c$ is a condensation point of $A$. Indeed, for
every neighborhood $U$ of $c$ there exists $n\in\mathbb{N}$ such
that $\left[a_{n},b_{n}\right]\subseteq U$ , and since $A\cap\left[a_{n},b_{n}\right]$
is large, the claim follows. 
\end{proof}
\begin{thm}[{Sierpi\'{n}ski, \cite{Sierpinski},\cite{Moore},\cite[Form 6]{TheDictionary}}]
\label{thm:from choice-exist 2 condensation points}
The following are equivalent:
\begin{enumerate}
\item $CUT (\mathbb{R})$.
\item Every large and bounded subset of $\mathbb{R}$ has at least two condensation points (equivalently, every strongly concentrated set is countable).
\item Every large subset of $\mathbb{R}$ has a condensation point.
\item For all $A\subseteq \mathbb{R}^n$ if $A\cap B$ is countable for every bounded $B\subseteq \mathbb{R}^n$, then $A$ is countable.
\end{enumerate}
\end{thm}
\begin{proof}
The equivalence 1 $\iff$ 4 holds by Theorem 5 in \cite{equivalent_principle_on_real_line}.

1 $\implies$ 2:

Suppose that
any union of countably many countable sets of real numbers is countable.
Let $A\subseteq\mathbb{R}$, be a large and bounded set. 
From Theorem \ref{thm:Existence-of-a condensation point} it follows
that there is $c\in\mathbb{R}$ which is a condensation point of $A$.
Let $\left(a_{n}\right)_{n\in\mathbb{N}}$, $\left(b_{n}\right)_{n\in\mathbb{N}}$
be sequences of real numbers, such that $a_{n}$ is strictly increasing
and $\lim_{n\to\infty}a_{n}=c$, and $b_{n}$ is strictly decreasing
and $\lim_{n\to\infty}b_{n}=c$.

If for every $n\in\mathbb{N}$ we have that $A\cap\left[a_{n},a_{n+1}\right]$
and $A\cap\left[b_{n+1},b_{n}\right]$ are both not large then $A$ is
the union of at most countably many sets of reals, each one of them
is at most countable, and therefore, $A$ is at most countable, a
contradiction to the assumption. Thus, there exists a natural number
$n$ such that $A\cap\left(\left[a_{n},a_{n+1}\right]\cup\left[b_{n+1},b_{n}\right]\right)$
is large. Thus, by Theorem \ref{thm:Existence-of-a condensation point}
there exists $c'\in\mathbb{R}$ which is a condensation point of $A\cap\left(\left[a_{n},a_{n+1}\right]\cup\left[b_{n+1},b_{n}\right]\right)$
and in particular of $A$. $c'\neq c$ because $c'\in[a_n,a_{n+1}]\cup[b_n+1,b_n]$ and  $c\notin[a_n,a_{n+1}]\cup[b_n+1,b_n]$.

2 $\implies$ 3:

We prove that the negation of 3 implies the negation of 2. 

Let $A^\star \subset \mathbb{R}$ a large set with no condensation point.
By Theorem \ref{thm:Existence-of-a condensation point} for every $a<b\in \mathbb{R}$ the set $\{x\in A^\star\mid a<x<y\}$ is not large.
Let us define $A=\{|x|\mid x\in A^\star\}$
$A$ is large with no condensation point, and every element in $A$ is  bigger then $-1$.
There is a function $f\colon \mathbb{R} \to \{x\in \mathbb{R}\mid 0<x<1\}$ which is an order isomorphism.
Let us define $B=\{y\in\mathbb{R}\mid 0<y<1,\exists x\in A,f(x)=y\}$.
$B$ is bounded.
$B$ is a large set because $f$ is bijection and thus $|B|=|A|$.
For every $r\in\mathbb{R}$ if $r\neq1$ then $r$ is not a condensation point of $B$
because for every $D\subseteq\mathbb{R}$ if $D$ is closed and $1\in \mathbb{R}\setminus D$
then $\{x\in A\mid f(x)\in D\}$ is not large set.

3 $\implies$ 1:  

We prove that the negation of 1 implies the negation of 3.

We assume there is an uncountable subset of $\mathbb{R}$ which this set is a result of a countable union of  countable sets. $|\mathbb{R}|=|\{x\in \mathbb{R}\mid 0<x<1\}|$ therefore there exists an uncountable set $A\subseteq\{x\in\mathbb{R}\mid 0<x<1\}$ and a sequence of pairwise-disjoint and countable sets $\langle A_n\rangle_{n\in\omega}$ such that $A=\bigcup_{n\in\omega} A_n$. For all $n\in\omega$ we define $B_n = \{x\in\mathbb{R}\mid x-n\in A_n\}$.
$|B_n|=|A_n|=\aleph_0$. Let us define $B=\bigcup_{n\in\omega}B_n$.
$|B|=|A|$ and therefore $B$ is uncountable. $B$ has no condensation points because every bounded subset of $B$ is either finite or countable.

\end{proof}
By Theorem \ref{thm:from choice-exist 2 condensation points}, $ZFC$ proves that any strongly concentrated set of reals is at most countable.

\begin{problem}
\label{problem} (Sierpi\'{n}ski)\cite{Sierpinski} Is it true that
one cannot prove, without using choice, that every bounded and large
set $A\subseteq\mathbb{R}$, has at least two condensation points?

In this paper we interpret this question as follows:

Does $ZF$ prove that every large and bounded set $A\subseteq\mathbb{R}$,
has at least two condensation points? Equivalently, does $ZF$ prove
that any strongly concentrated set of reals is at most countable?
\end{problem}

In the standard examples of failure of $CUT(\mathbb{R})$ such as the Feferman-Levy model (\cite{FefermanLevy1963}), the obtained
strongly concentrated set of reals is not well orderable. Yair Hayut asked the following: 
\begin{problem}
\label{prob:with well order}Is it true that one cannot
prove in $ZF$ that every bounded, well orderable and large set $A\subseteq\mathbb{R}$,
has at least two condensation points? 
\end{problem}
We will isolate two models of $ZF$. In both models there is a large
bounded subset of $\mathbb{R}$ with a unique condensation point.
In the first one, this set is well orderable, and in the second one
$\aleph_{1}$ is regular.

\section{Well ordered large strongly concentrated sets}\label{section: wo}

In this section we will show that there is a large well orderable strongly concentrated
set of reals if and only if $\aleph_{1}$ is singular and there is
an injection of $\aleph_{1}$ into the reals. 
\begin{thm}
\label{thm:well orderable large set to singular}
The following are equivalent:

\begin{itemize}
\item There is a well orderable strongly concentrated  set of real numbers.
\item $\aleph_1 < 2^{\aleph_0}$ (exist a one to one function from $\omega_1$ to $\mathbb{R}$) and $\mathrm{cf}\ (\aleph_1)=\aleph_0$.	
\end{itemize}	
\end{thm}
The conjunction of the following two lemmas implies the theorem. 

\begin{lem}
\label{lem: singular aleph1} Assume that there is a bounded, well
orderable set $A\subseteq\mathbb{R}$, with a unique condensation point.
Then $\mathrm{cf}\left(\aleph_{1}\right)=\aleph_{0}$ and $\left|A\right|=\aleph_{1}$.
In particular, there is an injection $f\colon\omega_{1}\to\mathbb{R}$. 
\end{lem}
\begin{proof}
Clearly, $A$ is uncountable, because $A$ has a condensation
point. Therefore, since $A$ can be well ordered, $\left|A\right|\geq\aleph_{1}$.

Let us show that there is $B\subseteq P_{\aleph_{1}}\left(\mathbb{R}\right)$
such that $\left|B\right|=\aleph_{0}$ and $A=\bigcup B$. This is
done by imitating the proof of Theorem \ref{thm:from choice-exist 2 condensation points}.

Namely, let $c$ be the unique condensation point of $A$. Let 
\[
B_{n}=A\setminus\left(c-\frac{1}{n},c+\frac{1}{n}\right)
\]
and define $B=\left\{ B_{n}\mid n\in\mathbb{N}\setminus\left\{ 0\right\} \right\} $.
If there is a natural number $n$ such that $B_{n}$ is large, then
$B_{n}$ has a condensation point. This condensation point cannot
be $c$, since $c$ is not in the closure of $B_{n}$.

Let us use the following lemma: 
\begin{claim}
The cardinality of a countable union of countable sets of ordinals
is at most $\aleph_{1}$. 
\end{claim}
\begin{proof}
Let $B$ be a set which is a countable union of countable sets of
ordinals. We claim that $\left|B\right|\leq\aleph_{1}$. Suppose otherwise.
Let $B$ be a counterexample. Passing to the cardinality of $B$,
we can replace it by an $\aleph_{\alpha}$ with $\alpha\geq2$.

Let us fix a countable sequence of countable subsets of $\aleph_{\alpha}$,
$\langle B_{n}^{\star}\mid n\in\omega\rangle$, such that $\aleph_{\alpha}=\bigcup_{n<\omega}B_{n}^{\star}$.
We define a sequence of sets 
\[
B_{n}=B_{n}^{\star}\setminus\left(\bigcup_{k<n}B_{k}^{\star}\right)
\]
for each $n\in\omega$.

The sets $\{B_{n}\mid n\in\omega\}$ are pairwise disjoint. Set $\beta_{n}=\mathrm{otp}\left(B_{n}\right)$,
for every $n<\omega$. $\beta_{n}<\omega_{1}$, since $B_{n}$ is
countable, and therefore so is $\beta_{n}$. Define by induction a
sequence of countable ordinals $\left\langle \gamma_{n}\mid n<\omega\right\rangle $
as follows:

$\gamma_{0}=\beta_{0}$, and for all $n<\omega$, $\gamma_{n+1}$
is the least ordinal $\gamma$ such that $\mathrm{otp}\left(\gamma\setminus\gamma_{n}\right)$
has order type $\beta_{n+1}$. Clearly, for every $n<\omega$, $\gamma_{n}$
is countable and uniquely determined. Set $\gamma^{\star}=\bigcup_{n<\omega}\gamma_{n}$.
Then $\gamma^{\star}\leq\aleph_{1}$.

Let us denote by $\pi_{X,Y}$ the unique order isomorphism between
sets of ordinals $X$,$Y$.

Finally, let us define a bijection $f\colon\aleph_{\alpha}\to\gamma^{\star}$
as follows: for every $\nu<\aleph_{\alpha}$ there exists a unique
$n^{\star}$ such that $\nu\in B_{n^{\star}}$. If $n^{\star}=0$,
set $f\left(\nu\right)=\pi_{B_{0},\gamma_{0}}\left(\nu\right)$. Otherwise,
$n^{\star}=n+1$ for some $n<\omega$, set $f\left(\nu\right)=\pi_{B_{n+1},\gamma_{n+1}\setminus\gamma_{n}}$.
\end{proof}
This concludes the proof of Lemma \ref{lem: singular aleph1}. 
\end{proof}
\begin{lem}
\label{thm: aleph1 singular implies strongly consentrated set} If
$\aleph_{1}$ is singular and there is an injection $g\colon\omega_{1}\to\mathbb{R}$
then there is a bounded, well orderable, set $A\subseteq\mathbb{R}$
with a unique condensation point. 
\end{lem}
\begin{proof}
Identify $\mathbb{R}$ with $^{\omega}2$ . By the assumption of the
theorem, there is a function 
\[
\nu\colon\omega\to\omega_{1}
\]
such that for all $n<m$, $\nu\left(n\right)<\nu\left(m\right)$ and
$\bigcup_{n\in\mathbb{N}}\nu\left(n\right)=\omega_{1}$. We define
a function $\rho\colon\omega_{1}\to\omega$ by 
\[
\rho\left(\alpha\right)=\min\left\{ n\in\omega\mid\nu\left(n\right)>\alpha\right\} .
\]

Let $g\colon\omega_{1}\to\mathbb{R}$ be an injection. Let us define
a function $f\colon\omega_{1}\to\mathbb{R}$ by: 
\[
f\left(\alpha\right)\left(n\right)=\begin{cases}
1 & \rho\left(\alpha\right)>n\\
0 & \rho\left(\alpha\right)=n\\
g\left(\alpha\right)\left(n-\rho\left(\alpha\right)-1\right) & \rho\left(\alpha\right)<n
\end{cases}.
\]

Thus the real number $f\left(\alpha\right)$ is obtained by adding
$\rho(\alpha)$ 1-s and a single zero at the beginning of $g\left(\alpha\right)$.
$f$ is an injection since for all $\alpha\in\beta\in\omega_{1}$,
$\rho\left(\alpha\right)\leq\rho\left(\beta\right)$.

If $\rho\left(\alpha\right)<\rho\left(\beta\right)$ then 
\[
f\left(\alpha\right)\left(\rho\left(\alpha\right)\right)=0\neq1=f\left(\beta\right)\left(\rho\left(\alpha\right)\right)
\]

and if $\rho\left(\alpha\right)=\rho\left(\beta\right)$ then since
$g$ is one to one there is some $n\in\omega$ such that $g\left(\alpha\right)\left(n\right)\ne g\left(\beta\right)\left(n\right)$.

Let $A$ be $\im f$.

$A\subseteq\mathbb{R}$ is a large set (since $|A|=\aleph_1$). By Theorem \ref{thm:Existence-of-a condensation point},
$A$ has a condensation point.

For every $y\in\mathbb{R}$ if there is $n\in\omega$ such that $y\left(n\right)=0$
then there is some $\alpha\in\omega_{1}$ such that for every $\beta\in\omega_{1}$,
$\beta>\alpha$ and every $k<n+2$, $f\left(\beta\right)\left(k\right)=1$.
Thus, $y$ is not a condensation point of $A$.
\end{proof}
The assumptions of lemma \ref{thm: aleph1 singular implies strongly consentrated set}
hold in a Feferman-L\'{e}vy model. Namely, let $V$ be a well founded
model of $ZFC$ such that $\aleph_{\omega}<2^{\aleph_{0}}$ (this
can be arranged, for example, by adding $\aleph_{\omega}$ Cohen reals).
Use the Feferman-L\'{e}vy construction over $V$ (See \cite[Chapter 10]{JechAC})
to get a model $M$ of $ZF$. $M\vDash\aleph_{1}^{M}=\aleph_{\omega}^{V}$.
In $M$, there is an injection $f\colon\omega_{1}\to\mathbb{R}$ and
$\aleph_{1}$ is a singular cardinal.

\section{Large Strongly Concentrated sets with regular $\aleph_{1}$}\label{section: regular aleph1}
By the previous section, if $\aleph_{1}$ is singular and injects
into the reals, then there is a large, well-orderable and strongly concentrated set. The existence of a large strongly concentrated set is consistent with the regularity of $\aleph_{1}$.
This statement for example holds in Sageev's Model, \cite{sagiv}.
In this section we represent other way to get a model with this feature.  
One notable difference between the method which is introduced in the previous section and the method that we introduce in this section that while the method of the previous section collapse all uncountable cardinals below $\beth_\omega$ to $\aleph_0$, the current method  preserves all cardinals above $\beth_1$ as cardinals.   

Let us start with a well founded model of $ZFC$, $W$. In particular,
$\left(2^{\aleph_{0}}\right)^{+}$ is a regular cardinal in $W$.

Let $V$ be a model of $ZFA+AC$ and let $A$ be the set of all atoms
in $V$. Let us assume that $\left|A\right|>\aleph_{0}$.
\begin{defn}
Let $\mathcal{S}$ to be the group of all bijection $\pi\colon A\to A$. 
\end{defn}
\begin{defn}
For $\pi\in S$ and $x\in V\setminus A$ we define $\pi\left(x\right)$ recursively
as 
\[
\pi(x)=\left\{ \pi\left(t\right)\mid t\in x\right\} .
\]
\end{defn}
\begin{defn}
For all $x\in V$ we define 
\[
sym_{S}\left(x\right)=\left\{ \pi\in S\mid\pi\left(x\right)=x\right\} .
\]
\end{defn}
\begin{defn}
For all $C\in P_{\aleph_{1}}\left(A\right)$ we define: 
\[
\mathcal{S}_{C}=\left\{ \pi\in\mathcal{S}\mid\forall a\in C,\pi\left(a\right)=a\right\} .
\]
\end{defn}
\begin{defn}
We define 
\[
F=\left\{ H\leq S\mid\exists C\in P_{\aleph_{1}}\left(A\right),\mathcal{S}_{C}\leq H\right\} 
\]
$F$ is a filter of subgroups over $S$. 
\end{defn}
\begin{defn}
We define 
\[
mys=\left\{ x\in V\mid sym_{S}\left(x\right)\in F\right\} .
\]
$mys$ is the class of all symmetric elements. 
 We define 
\[
\mathcal{B}=\left\{ x\in V\mid TC\left(x\right)\subseteq mys\right\} .
\]
$\mathcal{B}$ is the class of all hereditary symmetric elements. 
\end{defn}
By a well known theorem of Fraenkel (see \cite{JechSetTheory}) $\mathcal{B}$
is a model of $ZFA$.

\begin{defn}
For all $x\in V$ we define 
\[
St\left(x\right)=\left\{ C\in P_{\aleph_{1}}\left(A\right)\mid\mathcal{S}_{C}\subseteq sym_{S}(x)\right\} .
\]
\end{defn}
 Work in $\mathcal{B}$. 
\begin{defn}
We define a forcing 
\[
Q=\left\{ h:D\to\left\{ 0,1\right\} \mid D\in P_{\aleph_{0}}\left(\omega\right)\right\} .
\]
We say that $h_{0}$ is stronger than $h_{1}$ or equal to $h_{1}$
if and only if $\dom h_{1}\subseteq\dom h_{0}$

and $\forall d\in\dom h_{1},h_{0}\left(d\right)=h_{1}\left(d\right)$. 
\end{defn}
$Q$ is essentially the Cohen forcing. 
\begin{defn}
We define 
\[
I=\left\{ h\in Q\mid\exists n\in\dom h,\ h\left(n\right)=0\right\} .
\]
\end{defn}
\begin{defn}
\label{definition: Pn} Let $n\in\omega$.

We define $P_{n}$ to be the set of all functions $f\colon A\to Q$
such that: 
\begin{enumerate}
\item For all $a\in A$, $\dom\,f\left(a\right)=n$. 
\item $f^{-1}\left(I\right)\in P_{\aleph_{1}}\left(A\right)$. 
\item \label{requirement:big community} $\forall t\colon n\to\left\{ 0,1\right\} $,
there are infinitely many $a\in A$ such that $f\left(a\right)=t$. 
\end{enumerate}
\end{defn}
\begin{defn}
We define a forcing 
\[
P=\bigcup_{n\in\omega}P_{n}.
\]

We order $P$ by: 
\[
\forall f_{0},f_{1}\in P,f_{0}\leq f_{1}\Leftrightarrow\forall a\in A,f_{0}\left(a\right)\leq_{Q}f_{1}\left(a\right).
\]
\end{defn}
\begin{defn}
For all $n\in\omega$ we define 
\[
D_{n}=\bigcup_{k\in\omega\setminus n}P_{k}.
\]
\end{defn}
\begin{defn}
For all $C\in P_{\aleph_{1}}\left(A\right)$ we define 
\[
Z_{C}=\{f\in P\mid\forall a\in C,f(a)\in I\}..
\]
\end{defn}
Let $G$ be a generic filter for $P$. 
\begin{defn}
For all $a\in A$ we define 
\[
G_{a}=\left\{ h\in Q\mid\exists f\in G,\,f(a)=h\right\} .
\]
Let us define $g_{a}=\bigcup G_{a}$. 
\end{defn}
\begin{defn}
For all $C\in P_{\aleph_{1}}(A)$, let $g_{C}^{\star}\colon C\to\ ^{\omega}\{0,1\}$,
be the function $g_{C}^{\star}(c)=g_{c}$, for every $c\in C$. Let
$g_{C}$ be a name which is forced by the weakest condition to be
$g_{C}^{\star}$. 
\end{defn}
\begin{defn}
For all $C\in P_{\aleph_{1}}(A)$ we define $Res_{C}^{\star}\colon I\to\omega\cup\{\aleph_{0}\}$
which for all $i\in I$, $Res_{C}^{\star}(i)=|\{a\in A\setminus C\mid i\in g^{\star}(a)\}|$.
Let $Res_{C}$ a name which is forced by every condition to be $Res_{C}^{\star}$. 
\end{defn}
\begin{defn}
We define a function $num\colon P\to\omega$ by 
\[
num\left(f\right)=\min\left\{ n\in\omega\mid f\in D_{n}\right\} .
\]
\end{defn}
\begin{lem}
\label{lem:colaps bet one} $2^{\aleph_{0}}$ of $\mathcal{B}$ is
countable in $\mathcal{B}\left[G\right]$. 
\end{lem}
\begin{proof}
Work in $\mathcal{B}$. Fix a sequence $\left\langle \ell_{k}\mid k<\omega\right\rangle $
of injective functions from $\omega$ to $A$ with disjoint images. 
\begin{claim}
\label{claim: meeting a single function} For every $h\colon\omega\to\left\{ 0,1\right\} $
the following set is dense: 
\[
D_{h}=\left\{ f\in P\mid\exists n<num\left(f\right),m<\omega,\forall k<\omega,f\left(\ell_{m}\left(k\right)\right)(n)=h\left(k\right)\right\} .
\]
\end{claim}
\begin{proof}
Let $f^{\star}\in P$ and $n=num\left(f^{\star}\right)$. Let $T_{n}={}^{n}\{0,1\}$,
the set of all functions $t\colon n\to\left\{ 0,1\right\} $. Define
a function $F\colon\omega\to\mathcal{P}\left(T_{n}\right)$ as follows:
\[
F(k)=\left\{ t\in T_{n}\mid\exists a\in\im\ell_{k},f^{\star}\left(a\right)=t\right\} .
\]
Note that $F(k)=\im f^{\star}\circ\ell_{k}$, and hence it is never
empty. $F$ defines a partition of $\omega$ into finitely many pieces.
Hence at least one of them must be infinite. So there is $x\subseteq T_{n}$
and an infinite $Y\subseteq\omega$ such that for every $k\in Y$,
\[
F\left(k\right)=x.
\]
Which means, for every $t\in x$ and $k\in Y$ there is $a\in\im\ell_{k}$
such that 
\[
f^{\star}(a)=t.
\]
Let $k^{\star}=\min Y$. Note that for any $t\in T_{n}$ there are
infinitely many $a\in A\setminus\im\ell_{k^{\star}}$ such that $f^{\star}(a)=t$.
Extend $f^{\star}$ to a condition $f\in P_{n+1}$ as follows: for
all $a\in\im\ell_{k^{\star}}$, $f(a)(n)=h\left(\ell_{k}^{-1}(a)\right)$.
For elements in $A\setminus\im\ell_{k^{\star}}$ define $f$ such
that requirement \ref{requirement:big community} in Definition \ref{definition: Pn}
will be satisfied. This is possible, since for every $t\in T_{n}$
there are infinitely many members of $A$ which are not in $\im\ell_{k^{\star}}$
such that $f^{\star}(a)=t$.

\end{proof}
So, $G$ intersects each $D_{h}$.

Define $\varTheta\colon\omega\times\omega\rightarrow^{{\rm {onto}}}\left(\left\{ h\colon\omega\to\left\{ 0,1\right\} \right\} \right)^{\mathcal{B}}$,
as follows: $\varTheta\left(\left(n,m\right)\right)=h$ if and only
if $\exists f\in G\cap P_{n+1}$, $\forall k\in\omega,f\left(\ell_{m}\left(k\right)\right)\left(n\right)=h\left(k\right)$. 
\end{proof}
The following lemma follows from the proof of claim \ref{claim: meeting a single function}. 
\begin{lem}
\label{lem:to the next}For all $n\in\omega$ and $f^{\star}\in P_{n}$
there is $f\in P_{n+1}$ such that $f$ is stronger than $f^{\star}$. 
\end{lem}
\begin{lem}
\label{lem:order dense set}For all $n\in\omega$ the set $D_{n}$
is dense set in $P$. 
\end{lem}
\begin{proof}
We prove the lemma by induction on $n$.

For $n=0$ the claim is true by the definition of $D_{0}=P$.

We assume the validity of the claim for $n$. Let $f^{\star}\in P$,
by the induction hypothesis there exists $f^{+}\in D_{n}$ such that
$f^{+}$ is stronger than $f^{\star}$ or equal to $f^{\star}$.

If $num\left(f^{+}\right)>n$ we define $f=f^{+}$, and get that $f\in D_{n+1}$.

If $num\left(f^{+}\right)=n$ then by lemma \ref{lem:to the next}
there is a condition $f\in P_{n+1}$ stronger than $f^{+}$. Thus
$f\in D_{n+1}$ and $f$ is stronger than $f^{\star}$. 
\end{proof}
We conclude that $\forall a\in A,n\in\omega,\exists f\in G$ such
that $n\in\dom\,f\left(a\right)$.
\begin{lem}
\label{lem:no big for bet one plus}In $\mathcal{B}$, for all function
$\psi\colon P\to\left(2^{\aleph_{0}}\right)^{+}$, $\left|\im\,\psi\right|\leq2^{\aleph_{0}}$. 
\end{lem}
\begin{proof}
Let $C\in St\left(\psi\right)$.

We define a function 
\[
\vartheta_{0}\colon P\to\left\{ h_{0}\colon C\to Q\right\} 
\]
by 
\[
\vartheta_{0}\left(f\right)\left(c\right)\left(n\right)=f\left(c\right)\left(n\right).
\]

We define 
\[
\vartheta_{1}\colon P\to\left\{ h_{1}\colon I\to\omega\cup\{\aleph_{0}\}\right\} 
\]
by 
\[
\vartheta_{1}\left(f\right)\left(i\right)=\left|\left\{ a\in A\setminus C\mid f\left(a\right)=i\right\} \right|.
\]

We define 

\[
\vartheta\colon P\to\left\{ h_{0}\colon C\to Q\right\} \times\left\{ h_{1}\colon I\to\omega\cup\{\aleph_{0}\}\right\} 
\]
by 
\[
\vartheta\left(f\right)=\left(\vartheta_{0}\left(f\right),\vartheta_{1}\left(f\right)\right).
\]
\begin{lem}
For all $f_{0},f_{1}\in P$ if $\vartheta\left(f_{0}\right)=\vartheta\left(f_{1}\right)$
then $\psi\left(f_{0}\right)=\psi\left(f_{1}\right)$. 
\end{lem}
\begin{proof}
We define 
\[
\Delta_{0}=\left\{ a_{0}\in A\setminus C\mid f_{0}(a_{0})\in I\right\} 
\]
and 
\[
\Delta_{1}=\left\{ a_{1}\in A\setminus C\mid f_{1}(a_{1})\in I\right\} .
\]
$\Delta_{0},\Delta_{1}$ are finite or countable and for all $i\in I$,
\[
\left|\left\{ a_{0}\in\Delta_{0}\mid f_{0}(a_{0})=i\right\} \right|=\left|\left\{ a_{1}\in\Delta_{1}\mid f_{1}(a_{1})=i\right\} \right|
\]
since $\vartheta_{1}(f_{0})=\vartheta_{1}(f_{1})$. Let $D\subseteq A\setminus(C\cup\Delta_{0}\cup\Delta_{1})$
be countable. Let us define a permutation of $D$, $\varTheta$, such
that for all $a\in D$, 
\[
f_{0}\left(a\right)=f_{1}\left(\varTheta\left(a\right)\right).
\]

Let us extend $\varTheta$ to a bijection $\Lambda\colon A\to A$
by defining $\Lambda(a)=a$ for all $a\notin D$.

In particular, for all $a\in C$, $\varLambda\left(a\right)=a$. Therefore:
\[
\varLambda\left(\psi\right)=\psi.
\]

\[
\varLambda\left(f_{0}\right)=f_{1}.
\]

and 
\[
\psi\left(f_{1}\right)=\Lambda\left(\psi\right)\left(\varLambda\left(f_{0}\right)\right)=\varLambda\left(\psi\left(f_{0}\right)\right)=\psi\left(f_{0}\right).
\]
\end{proof}
We conclude that $|\im\psi|\leq|\left\{ h_{0}\colon C\to Q\right\} |\cdot|\left\{ h_{1}\colon I\to\omega+1\right\} |=2^{\aleph_{0}}$. 
\end{proof}
\begin{lem}
\label{lem:bet one plus of B is regular} In $\mathcal{B}$, for all
$P$-name $\varUpsilon$ and $f^{\star}\in P$ if 
\[
f^{\star}\Vdash\varUpsilon=\left\langle \alpha_{n}\mid n\in\omega\right\rangle ,\ \alpha_{n}\in\left(\left(2^{\aleph_{0}}\right)^{+}\right)^{\mathcal{B}},
\]
then $f^{\star}\Vdash\bigcup_{n\in\omega}\alpha_{n}\in\left(\left(2^{\aleph_{0}}\right)^{+}\right)^{\mathcal{B}}$
. 

\end{lem}
\begin{proof}
For all $n\in\omega$ we define

\[
P^{\Upsilon_{n}}=\left\{ f\in P\mid\exists\alpha\in\left(\left(2^{\aleph_{0}}\right)^{+}\right)^{\mathcal{B}},f\Vdash\alpha_{n}=\check{\alpha}\right\} 
\]
and

\[
\Psi_{n}\colon P^{\Upsilon_{n}}\to\left(\left(2^{\aleph_{0}}\right)^{+}\right)^{\mathcal{B}}
\]
such that for all $f\in P^{\Upsilon_{n}}$, 
\[
f\Vdash\alpha_{n}=\varPsi_{n}\left(f\right).
\]

By lemma \ref{lem:no big for bet one plus} $\left|\im\varPsi_{n}\right|\leq2^{\aleph_{0}}$
(in $\mathcal{B}$), since by the regularity of $\left(2^{\aleph_{0}}\right)^{+}$,
\[
\alpha_{n}\leq\beta_{n}=\sup\im\varPsi_{n}<\left(2^{\aleph_{0}}\right)^{+}.
\]

and 
\[
\bigcup_{n\in\omega}\beta_{n}<\left(2^{\aleph_{0}}\right)^{+}.
\]
\end{proof}
\begin{lem}
\label{lem:bet one plus until regular}$\aleph_{1}^{\mathcal{B}\left[G\right]}=\left(\left(2^{\aleph_{0}}\right)^{+}\right)^{\mathcal{B}}$
and it is a regular cardinal in the generic extension. 
\end{lem}
\begin{proof}
By lemma \ref{lem:colaps bet one}, $\aleph_{1}^{\mathcal{B}\left[G\right]}\geq\left(\left(2^{\aleph_{0}}\right)^{+}\right)^{\mathcal{B}}$.

By lemma \ref{lem:bet one plus of B is regular}, $\aleph_{1}^{\mathcal{B}\left[G\right]}\leq\left(\left(2^{\aleph_{0}}\right)^{+}\right)^{\mathcal{B}}$
is a regular cardinal in $\mathcal{B}\left[G\right]$. 
\end{proof}
\begin{lem}
\label{lem:zero on countable atoms} For all $C\in P_{\aleph_{1}}(A)$
the set $Z_{C}$ is dense. 
\end{lem}
\begin{proof}
Let $f^{\star}$ be a condition in $P$. Let us denote $n=num\left(f^{\star}\right)$
and 
\[
E=\left\{ a\in C\mid\forall k\in n,f^{\star}\left(a\right)\left(k\right)=1\right\} .
\]

For all $t:n\to\left\{ 0,1\right\} $ we choose $\ell_{t}\subseteq\left\{ a\in A\mid f^{\star}\left(a\right)=t\right\} $
such that $\left|\ell_{t}\right|=\aleph_{0}$, $\{a\in A\mid f^{\star}\left(a\right)=t\}\setminus\ell_{t}$
is infinite, and $\ell=\left(\bigcup_{t:n\to\left\{ 0,1\right\} }\ell_{t}\right)\cup E$.

We define $f\in P$ by: 
\[
f\left(a\right)\left(k\right)=\begin{cases}
f^{\star}\left(a\right)\left(k\right) & k\neq n\\
0 & k=n,a\in\ell\\
1 & k=n,a\notin\ell
\end{cases}
\]

For all $a\in A,\dom f^{\star}\left(a\right)\cap\left\{ n\right\} =\phi$
and by first line of the definition of $f$ we get that if $k\in\dom\,f^{\star}\left(a\right)$then
($k\neq n$) $f\left(a\right)\left(k\right)=f^{\star}\left(a\right)\left(k\right)$
therefore $f$ is stronger than $f^{\star}$.

For all $a\in C$ if $a\notin E$ then $\exists k\in n$ ($k\neq n$)
such that $f\left(a\right)\left(k\right)=f^{\star}\left(a\right)\left(k\right)=0$

and if $a\in E$ then $f\left(a\right)\left(n\right)=0$. 
\end{proof}
\begin{thm}
$\forall a,b\in A,a\neq b\Rightarrow\exists f\in G,n\in\omega,f\left(a\right)\left(n\right)=0,f\left(b\right)\left(n\right)=1$. 
\end{thm}
\begin{proof}
Let $f^{\star}\in P$ by definition exist $n\in\omega$ such that
$f^{\star}\in P_{n}$, we chose $f^{+}\in P_{n+1}$ such that $f^{+}$
stronger than $f^{\star}$(exist by lemma \ref{lem:to the next}).

We define $f:A\to Q$ by 
\[
f\left(c\right)\left(k\right)=\begin{cases}
f^{+}\left(c\right)\left(k\right) & c\in A-\left\{ a,b\right\} \vee k\neq n\\
0 & c=a\wedge k=n\\
1 & c=b\land k=n
\end{cases}
\]
\end{proof}
\begin{thm}
For all $n\in\omega$ the set $\left\{ a\in A:\exists f\in G,k\in n,f\left(a\right)\left(k\right)=0\right\} $
is a countable set in the ground model. 
\end{thm}
\begin{proof}
By lemma \ref{lem:order dense set} the set $D_{n}$ is dense, then
exist $f^{\star}\in D_{n}\cup G$, by definition of the forcing $\left|\left\{ a\in A\mid\exists k\in n,f^{\star}\left(a\right)=0\right\} \right|=\aleph_{0}$.

Let $a\in A$ and $f\in G$ and if exist $k\in n$ such that $f\left(a\right)\left(k\right)=0$
then ($f\parallel f^{\star}$and $k\in\dom f\left(a\right)$) $f^{\star}\left(a\right)\left(k\right)=0$. 
\end{proof}
\begin{thm}
$A$ is not countable in the generic extension. 
\end{thm}
\begin{proof}
Let $f^{\star}$ be a condition in $P$ and let $h$ be a $P$-name such that $f^{\star}\Vdash h:\omega\to A$, and let $C\in St\left(h\right)$.

Let $G\subset P$ generic such that $f^{\star}\in G$. By lemma \ref{lem:zero on countable atoms}
there exists a condition $f\in Z_{C}\cap G$ such that $f$ is stronger than $f^{\star}$. 
\begin{claim}
$f\Vdash \mathrm{Im}\ h\subseteq X_{f}=\left\{ a\in A\mid\exists n\in\omega,\ f\left(a\right)\left(n\right)=0\right\}$. 
\end{claim}
\begin{proof}
Suppose otherwise. 

By the definition of $P$, $X_{f}$ is at most countable and in particular, $X_f \neq A$. By the assumption, there is $a\in A\setminus X_{f}$, $k,n\in\omega$ and
$f^{+}\in D_{n}$ stronger than $f$ such that $f^{+}\Vdash h\left(k\right)=a$.

By the definition of $P$ the set \[S_{f^{+},a}=\left\{ b\in A\mid f^{+}\left(b\right)=f^{+}\left(a\right)\right\}\]
is an infinite set.

We claim that $S_{f^{+},a}\cap X_{f}=\emptyset$. For all $b\in S_{f^{+},a}$ and for
all $\ell$ if $\ell\in num(f)$ then $f(b)(\ell)=f(a)(\ell)=1$ since $a\notin X_f$.
Therefore $f^{+}(b)(\ell)=1$ and for all $c\in X_{f}$
exists $\ell\in num\left(f\right)$ such that $f(c)(\ell)=0$.
Thus $f^{+}(c)(\ell)=0$. 

Since $C \subseteq X_f$, we conclude that $S_{f^{+},a}\cap C=\emptyset$.

We define $t\in \mathcal{S}$ by 
\[
t\left(c\right)=\begin{cases}
b & c=a\\
a & c=b\\
c & c\notin\left\{ a,b\right\} 
\end{cases}.
\]


$t$ is an automorphism of $V$ and since
\[
f^{+}\Vdash h\left(k\right)=a
\]
we get that
\[
t\left(f^{+}\right)\Vdash t\left(h\right)\left(t\left(k\right)\right)=t\left(a\right).
\]
Moreover, since 
\begin{itemize}
\item $t\left(f^{+}\right)=f^{+}$. 
\item $t\left(h\right)=h$. 
\item $t\left(k\right)=k$ 
\end{itemize}
we conclude that
\[
f^{+}\Vdash h\left(k\right)=b
\]
contradicting the fact that $h$ is a function.
\end{proof}
Working in $V$, we conclude that $\mathrm{Im}\ h$ is forced by $f^+$ to be a subset of the countable set $X_{f}$. In particular, $\mathrm{Im}\ h \neq A$.
\end{proof}


By the general theory of $ZFA$, there is a model of $ZF$ with similar
properties (see \cite{JechSochorapplications,JechSochoronmodel}).
For completeness, let us describe a concrete way to obtain such a model
of $ZF$ in our case: 
\begin{defn}
In $\mathcal{B}\left[G\right]$ we define
\begin{itemize}
\item $\mathcal{C}_{0}=\emptyset$.
\item For a successor ordinal $\alpha = \beta + 1$, $\mathcal{C}_{\alpha}=\mathcal{P}(\mathcal{C}_{\beta})$.
\item For a limit ordinal $\alpha$, $\mathcal{C}_{\alpha}=\bigcup_{\beta\in\alpha}\mathcal{C}_{\beta}$.
\end{itemize}
Let \[\mathcal{C}=\left\{ x\in \mathcal{B}[G]\mid\exists\alpha\in Ord,x\in\mathcal{C}_{\alpha}\right\}.\] 
\end{defn}
\begin{thm}\label{thm:large set with regular}
It is consistent with $ZF$ that $\aleph_1$ is regular and there is a large set $A^\star\subseteq \mathbb{R}$ which is bounded and has a single condensation point.
\end{thm}
\begin{proof}
Let $A^\star = \{g_a \mid a \in A\}$. Note that $A^\star \in \mathcal{C}$.
	
$\mathcal{C}\models ZF$. Working in $\mathcal{C}$, there exists a set $A^{\star}\subseteq\mathbb{R}$
such that $A^{\star}\subseteq\left[0,1\right]$, $\left|A^{\star}\right|>\aleph_{0}$,
and for all $r\in\mathbb{R}$ if $r\in\left(0,1\right)$ then $\left|A^{\star}\cap [0,r]\right|=\aleph_{0}$.

Thus, $A^\star\subseteq\mathbb{R}$ is bounded and large and the point $1$
is the unique condensation point of $A^\star$.

By lemma \ref{lem:bet one plus until regular}, $\left(\aleph_{1}\right)^{\mathcal{C}}=\left(\aleph_{1}\right)^{\mathcal{B}\left[G\right]}=\left(\left(2^{\aleph_{0}}\right)^{+}\right)^{\mathcal{B}}$
is a regular cardinal in $\mathcal{B}[G]$ and therefore also in $\mathcal{C}$. 
\end{proof}
\begin{corollary*}
It is consistent that exist some set $A\subseteq \mathbb{R}$ with unique condensation point
and any set $A\subseteq \mathbb{R}$ with unique condensation point not have a well order.
\end{corollary*}
\begin{proof}
By theorem \ref{thm:large set with regular} it is consistent that exist some set $A\subseteq \mathbb{R}$ with unique condensation point and $\aleph_1$ is regular and by theorem \ref{thm:well orderable large set to singular} if $\aleph_1$ is regular then any set $A\subseteq \mathbb{R}$ with unique condensation point not have a well order.
\end{proof}

\providecommand{\bysame}{\leavevmode\hbox to3em{\hrulefill}\thinspace}
\providecommand{\MR}{\relax\ifhmode\unskip\space\fi MR }
\providecommand{\MRhref}[2]{%
	\href{http://www.ams.org/mathscinet-getitem?mr=#1}{#2}
}
\providecommand{\href}[2]{#2}


\end{document}